\documentclass[11pt]{amsart}

\usepackage{amsmath,amssymb,amsthm}
\usepackage{comment}
\usepackage[unicode,breaklinks=true,colorlinks=true]{hyperref}
\usepackage[dvipsnames]{xcolor}

\usepackage[top=1in, bottom=1in, left=1in, right=1in, marginparwidth=1in, marginparsep=0.1in]{geometry}

\numberwithin{equation}{section}
\newtheorem{theorem}{Theorem}[section]
\newtheorem{proposition}[theorem]{Proposition}

\theoremstyle{remark}

\definecolor{darkblue}{rgb}{0,0,0.7}

\renewcommand{\div}{\mathop{\mathrm{div}}}

\newcommand{\donothing}[1]{{}}

\makeatletter
\newcommand{\xRightarrow}[2][]{\ext@arrow 0359\Rightarrowfill@{#1}{#2}}
\makeatother

\let\OLDthebibliography\thebibliography
\renewcommand\thebibliography[1]{
	\OLDthebibliography{#1}
	\setlength{\parskip}{1pt}
	\setlength{\itemsep}{1pt plus 0.3ex}
}

\title[An inverse problem for the active scalar equations]{A Calder\'on type inverse problem for the active scalar equations with fractional dissipation}
\author{Li Li and Weinan Wang}
\date{}

\begin{document}
	
	\begin{abstract}
		
		In this paper, we are interested in an inverse problem for the active scalar equations with fractional dissipation on the torus. {We perform a second order linearization to relate our model to the linear fractional diffusion equation. Our approach to solving the inverse problem relies on nonlocal phenomena such as the unique continuation property of the fractional Laplacian and its associated Runge approximation property. A remarkable feature of our model is that the divergence-free structure in the nonlinear term plays an important role in both forward and
			inverse problems.}
	\end{abstract}
	
	\maketitle
	\section{Introduction}\label{intro}
	Active scalar equations describe a number of physical phenomena that arise in fluid dynamics. Roughly speaking, they are partial differential equations where the evolution in time of a scalar quantity is governed by the motion of the fluid in which the velocity itself varies with this scalar quantity \cite{FSW}. Typical examples include the Navier-Stokes equations, the Euler equations, the Boussinesq equations, and the porous media equation (see for instance, \cite{BE, NSE}).
	
	In this paper, we are interested in active scalar equations with fractional dissipation and a divergence-free drift velocity. This kind of equations appear in a number of physical regimes, such as stratification, rapid rotation, hydrostatics and geostrophic balance.
	We will focus on the torus $\mathbb{T}^2$ and the fractional power $\alpha\in
	(\frac{1}{2}, 1)$. More precisely, we consider the following initial value problem for the dissipative
	active scalar equation
	\begin{equation}\label{ASE}
		\left\{
		\begin{aligned}
			\partial_t \theta + u\cdot \nabla \theta +(-\Delta)^\alpha \theta &= f,\quad \,\,\, (x,t)\in \mathbb{T}^2\times (0, T),\\
			\theta(0)&= 0,\quad \,\,\,x\in \mathbb{T}^2.\\
		\end{aligned}
		\right.
	\end{equation}
	Here the vector-valued $u:= \mathcal{R}\theta$ has the form
	\begin{equation}\label{u1u2}
		(u_1, u_2)= (-\partial_2(\mathcal{K}\theta), \partial_1(\mathcal{K}\theta)),
	\end{equation}
	where $\mathcal{K}$ is a multiplier of order $-1$ (e.g. the Riesz potential $(-\Delta)^{-\frac{1}{2}}$). 
	
	{More precisely, we require that
		\begin{equation}\label{Ass1}
			K\in L^1(\mathbb{T}^2)\cap C^\infty(\mathbb{T}^2\setminus\{0\}).
		\end{equation}
		Its Fourier coefficients satisfy
		\begin{equation}\label{Ass2}
			c |k|^{-1}\leq \hat{K}(k)\leq C |k|^{-1},\qquad k\neq 0
		\end{equation}
		for some positive constants $c, C$.  Moreover,
		\begin{equation}\label{Ass3}
			\mathcal{R}:\, L^p(\mathbb{T}^2)\to L^p(\mathbb{T}^2)
		\end{equation}
		is bounded for $p\in (1, \infty)$. Here $K$ denotes the kernel corresponding to the multiplier $\mathcal{K}$, i.e.
		$$\mathcal{K}\theta(x)= K*\theta(x)= \int_{\mathbb{T}^2}K(x-y)\theta(y)\,\mathrm{d}y.$$}
	
	This model includes the surface quasi-geostrophic (SQG) equation (see \cite{SQG}), which corresponds to the vorticity equation of the 2D Euler equations in fluid dynamics.
	
	Given a nonempty open set $W\subset \mathbb{T}^2$ which can be arbitrarily small, we are interested in the determination of information on the operator $\mathcal{R}$ in $W^e:= \mathbb{T}^2\setminus\Bar{W}$, based on the observation on $\theta$ as well as the divergence-free drift velocity $u$ in $W$. 
	More precisely, we consider the following source to solution map
	\begin{equation}\label{stos}
		L_\mathcal{R}: f\,{\mapsto}\,(\theta|_{W\times (0, T)}, u|_{W\times (0, T)}),\qquad f\in C^\infty_c(W\times (0, T)
		).
	\end{equation}
	
	Our main result is the following theorem.
	\begin{theorem}\label{main}
		Let $T> 0$ which can be arbitrarily small. Suppose that $\mathcal{R}_1$ and $\mathcal{R}_2$ satisfy the assumptions (\ref{Ass1}-\ref{Ass3}).
		If $L_{\mathcal{R}_1}= L_{\mathcal{R}_2}$, then 
		$$\mathcal{R}_1g|_{W^e}= \mathcal{R}_2g|_{W^e},\qquad g\in 
		C^\infty_c(W^e).$$
	\end{theorem}
	{We remark that we are unable to extend the determination to the whole domain $\mathbb{T}^2$ mainly because we can obtain the density of solutions only in the exterior of $W$ (see Proposition \ref{RAP} in Section 5 later).}
	
	\subsection{Connection with earlier literature}
	The mathematical study of Calder\'on type inverse problems for space-fractional operators dates back to \cite{ghosh2020calderon}, where the authors considered the exterior Dirichlet problem for the fractional Schr\"odinger operator $(-\Delta)^s+ q$. They exploited the nonlocal features of the fractional operator to uniquely determine the potential $q$ inside the domain from partial knowledge of the Dirichlet-to-Neumann map which is an exterior measurement.
	We refer readers to \cite{ghosh2020uniqueness, covi2022higher, baers2024instability} and the references therein for extensions and variants of the fractional Calder\'on problem.
	
	It is well-known that the multiple-fold linearization can be useful in solving inverse problems for nonlinear equations. This technique was first introduced in \cite{kurylev2018inverse} and applied to local equations (see for instance, \cite{feizmohammadi2020inverse, krupchyk2020remark, uhlmann2021inverse}). Later the multiple-fold linearization was applied to fractional equations including
	the fractional semilinear elliptic equation (see \cite{lai2022inverse}), the nonlinear fractional magnetic equation (see \cite{lai2023inverse}), the fractionally damped wave equation (see \cite{li2023inverse}) and the nonlinear fractional elastic equation (see \cite{li2023elas}). 
	
	Our study in this paper can be viewed as a continuation of this thread. To the best knowledge of the authors, there is no existing literature on the study of the Calder\'on problem for the type of active scalar equations with fractional dissipation. One remarkable feature of our model is that the divergence-free structure in the
	nonlinear term brings convenience to both forward and inverse problems.
	Another remarkable feature is that we aim at determining the information on the nonlocal operator $\mathcal{R}$, rather than the usual local variable coefficients in the fractional Calder\'on problem.
	
	\subsection{Organization}\label{org}
	The rest of this paper is organized in the following way. In Section \ref{pre}, we will provide preliminary knowledge. In Section \ref{well}, we will state the well-posedness result for equation \eqref{ASE}, preparing for the linearization argument later. In Section \ref{lin}, we will rigorously perform the first and second order linearizations. In Section \ref{runge}, we will derive a Runge approximation property based on the the unique continuation property of the fractional Laplacian. In Section \ref{proof}, we will prove the main theorem. We will sketch the proof of the well-posedness of the forward problem in Section \ref{appen}.
	
	\section{Preliminaries}\label{pre}
	Throughout this paper, we focus on the space dimension $n=2$. We will write ``$n$" rather than ``$2$" when the definitions and propositions hold for general $n$. We use $\mathbb{T}^n:= \mathbb{R}^n/\mathbb{Z}^n$ to denote the $n$-torus. We also identify $\mathbb{T}^n$ with the cube
	$[-\frac{1}{2}, \frac{1}{2})^n$ for the purpose of integration.
	
	\subsection{Sobolev spaces}
	The (formal) Fourier series of $u$ is denoted by
	$$\sum_{k\in\mathbb{Z}^n}\hat{u}(k)e^{2\pi ix\cdot k},$$
	where $\hat{u}(k)$ is the Fourier coefficient. We have the (inhomogeneous) Sobolev space
	$$H^r(\mathbb{T}^n) =  \{u \in \mathcal{D}'(\mathbb{T}^n): \|u\|^2_{H^r}:=\sum_{k\in\mathbb{Z}^n} (1+|k|^{2})^r |\hat{u}(k)|^2 < +\infty\},$$
	and the homogeneous Sobolev space
	$$\dot{H}^r(\mathbb{T}^n) =  \{u \in \mathcal{D}'(\mathbb{T}^n): \|u\|^2_{\dot{H}^r}:=\sum_{k\in\mathbb{Z}^n, k\neq 0} |k|^{2r} |\hat{u}(k)|^2 < +\infty\}.$$
	For $u$ satisfying $\hat{u}(0)= 0$, the two norms are equivalent, and we do not distinguish
	$\dot{H}^r$ from $H^r$ in this case.
	
	For $r\geq 0$, the fractional Laplacian is defined by
	$$(-\Delta)^r u= \sum_{k\in\mathbb{Z}^n} |2\pi k|^{2r} \hat{u}(k)e^{2\pi ix\cdot k}.$$
	
	\subsection{Riesz transforms}
	To ensure the well-posedness of the forward problem and manipulate the integral identity in the inverse problem part, we need to impose the assumptions (\ref{Ass1}-\ref{Ass3}) on the operator 
	$\mathcal{R}$. A typical example of our model is the surface quasi-geostrophic (SQG) equation. In this case, the components of $\mathcal{R}$ are the Riesz transforms.
	
	It is well-known (see \cite{stein1970singular}) that in $\mathbb{R}^n$, the Riesz transforms can be defined by the singular integrals
	$$R_jf(x):= c_n\,\mathrm{p.v.}\int_{\mathbb{R}^n}
	\frac{(x_j- y_j)f(y)}{|x-y|^{n+1}}\,\mathrm{d}y,$$
	and they are bounded on $L^p(\mathbb{R}^n)$
	($1< p< \infty$).
	We can also interpret the Riesz transforms as the composition of the first order differential operators and the Riesz potential, i.e. $R_j= \partial_j (-\Delta)^{-\frac{1}{2}}$, where
	$$(-\Delta)^{-\frac{1}{2}}f(x):=
	c_n\int_{\mathbb{R}^n}
	\frac{f(y)}{|x-y|^{n-1}}\,\mathrm{d}y.$$
	
	In our torus case, the components of $\mathcal{R}$ are the periodic Riesz transforms.
	More precisely, we can write
	$\mathcal{R}\theta=
	(-R_2\theta, R_1\theta)$,
	where
	$$R_j\theta:= \sum_{k\neq 0,\, k\in\mathbb{Z}^2}-\frac{ik_j}{|k|}\hat{\theta}(k)e^{2\pi ix\cdot k},\qquad j=1, 2.$$
	It is well-known (see Theorem 2.17 in Chapter 7 in \cite{SteinIntro}) that 
	$\mathcal{R}$ is bounded on $L^p(\mathbb{T}^2)$
	($1< p< \infty$), and the corresponding $\mathcal{K}$ has the kernel $K\in L^1(\mathbb{T}^2)\cap C^\infty(\mathbb{T}^2\setminus\{0\})$ with the Fourier coefficients $\hat{K}(0)= 0$ and $\hat{K}(k)= |k|^{-1}$
	for $k\neq 0$. Moreover, $K$ has the same singularity as $|x|^{1-n}$ at the origin.
	
	\section{Well-posedness}\label{well}
	The proposition below essentially
	follows from Theorem 3.5 and Theorem 3.7 in \cite{Resnick1996}, which are originally stated for the SQG equation. The same proof works for our more general model, since the argument does not rely the explicit expression of $\hat{K}(k)$, but relies on the $L^p$-boundedness of $\mathcal{R}$ and the fact that the multiplier $\mathcal{K}$ commutes with differential operators and the fractional Laplacian. We will sketch the proof of the following proposition in Appendix (see Section \ref{appen}) to avoid redundancy.
	
	\begin{proposition}\label{p1}
		Let $\frac{1}{2}< \alpha< 1$ and $s> 0$. Suppose $0< \frac{1}{q}< \alpha- \frac{1}{2}$. Suppose 
		$$f\in L^1(0, T; L^q(\mathbb{T}^2))\cap  L^2(0, T; H^{s-\alpha}(\mathbb{T}^2))$$ and $\int_{\mathbb{T}^2}f(t, x)\,\mathrm{d}x= 0$ for $t\leq T$. Then
		(\ref{ASE}) has a unique solution $\theta$ satisfying
		$\int_{\mathbb{T}^2}\theta(t, x)\,\mathrm{d}x= 0$ for $t\leq T$ and
		$$\theta\in L^\infty(0, T; L^q(\mathbb{T}^2))\cap  L^\infty(0, T; H^{s}(\mathbb{T}^2))\cap  L^2(0, T; H^{s+\alpha}(\mathbb{T}^2)).$$
		Moreover, we have
		$$\|\theta(t)\|_{L^q}\leq Q_t(f):= \int^t_0\|f(\tau)\|_{L^q}\,\mathrm{d}\tau,$$
		\begin{equation}\label{Resnick}
			\|\theta(t)\|^2_{H^s}+ \int^t_0 \|\theta(\tau)\|^2_{H^{s+\alpha}}\,\mathrm{d}\tau
			\leq C_2\int^t_0 \|f(\tau)\|^2_{H^{s-a}}e^{C_1(t-\tau)Q^{\frac{\alpha}{\alpha-1/2-1/q}}_t(f)}\,\mathrm{d}\tau
		\end{equation}
		for $t\leq T$.
	\end{proposition}
	For general $f\in L^1(0, T; L^q(\mathbb{T}^2))\cap  L^2(0, T; H^{s-\alpha}(\mathbb{T}^2))$, we consider
	$$\tilde{f}:= f- \int_{\mathbb{T}^2}f(t, x)\,\mathrm{d}x,\qquad \tilde{\theta}:= \theta- \int^t_0\int_{\mathbb{T}^2}f(\tau, x)\,\mathrm{d}x\mathrm{d}\tau.$$ 
	Then we have 
	$\int_{\mathbb{T}^2}\tilde{f}(t, x)\,\mathrm{d}x= 0$ for $t\leq T$ and
	$$\partial_t \tilde{\theta} + \mathcal{R}\tilde{\theta}\cdot \nabla \tilde{\theta} +(-\Delta)^\alpha \tilde{\theta} = \tilde{f},$$
	and we can apply (\ref{Resnick}) to $\tilde{f}, \tilde{\theta}$.
	
	Let $\theta_f$ be the solution corresponding to the source $f$. The explicit integral on the right hand side of (\ref{Resnick}) will not be used in later sections, and we are only interested in the behavior of $\theta_{\epsilon f_0}$ for the small parameter $\epsilon$. More precisely, we note that for fixed $\alpha, s$, (\ref{Resnick}) implies 
	\begin{equation}\label{Resnickeps}
		\|\theta_{\epsilon f_0}\|_{L^\infty(0, T; H^{s}(\mathbb{T}^2))}+ \|\theta_{\epsilon f_0}\|_{L^2(0, T; H^{s+\alpha}(\mathbb{T}^2))}\leq C_{T, f_0}\epsilon
	\end{equation}
	for $|\epsilon|< 1$ and $f_0\in C^\infty_c(\mathbb{T}^2\times (0, T))$.
	
	\section{Linearization}\label{lin}
	To prepare for the proof of the main theorem, we will rigorously perform the first and second order linearizations in this section. The advantage of the second order linearization is to obtain products of the solutions of the linearized equations from the original nonlinear equation, which will enable us to use the density result for linear equations in solving the inverse problem.
	
	The following well-posedness result for the linear fractional diffusion equation will be used to prove our linearization results. {Since it follows from the standard Galerkin approximations, we will not provide a detailed proof to avoid redundancy. The detailed arguments should be analogues of the ones in Section 2.1 in \cite{ruland2020quantitative} and Section 7.1.3 in \cite{evans1998partial}.}
	
	\begin{proposition}\label{p2}
		Let $f\in L^2(0, T; H^{-\alpha}(\mathbb{T}^2))$. Then
		\begin{equation}\label{Linfrac}
			\left\{
			\begin{aligned}
				\partial_t u + (-\Delta)^\alpha u &= f,\quad \,\,\, (x,t)\in \mathbb{T}^2\times (0, T),\\
				u(0)&= 0,\quad \,\,\,x\in \mathbb{T}^2\\
			\end{aligned}
			\right.
		\end{equation}  
		has a unique solution $u$ satisfying
		$$u\in L^2(0, T; H^{\alpha}(\mathbb{T}^2))\cap C([0, T]; L^{2}(\mathbb{T}^2)),\qquad 
		\partial_t u\in L^2(0, T; H^{-\alpha}(\mathbb{T}^2)),$$  
		and we have 
		\begin{equation}\label{H-1est}
			\|u\|_{L^2(0, T; H^{\alpha}(\mathbb{T}^2))}+\|u\|_{C([0, T]; L^{2}(\mathbb{T}^2))}\leq C\|f\|_{L^2(0, T; H^{-\alpha}(\mathbb{T}^2))}.
		\end{equation}
		Moreover, if $f\in L^2(0, T; L^2(\mathbb{T}^2))$, then $u$ satisfies
		$$u\in L^2(0, T; H^{2\alpha}(\mathbb{T}^2))\cap L^\infty(0, T; H^{\alpha}(\mathbb{T}^2)),\qquad 
		\partial_t u\in L^2(0, T; L^2(\mathbb{T}^2)),$$  
		and we have 
		\begin{equation}\label{L2est}
			\|u\|_{L^2(0, T; H^{2\alpha}(\mathbb{T}^2))}+\|u\|_{L^\infty(0, T; H^{\alpha}(\mathbb{T}^2))}\leq C'\|f\|_{L^2(0, T; L^2(\mathbb{T}^2))}.
		\end{equation}
	\end{proposition}
	
	
	Let $f_1, f_2\in C^\infty_c(\mathbb{T}^2\times (0, T))$. We use $w_j$ ($j= 1,2$) to denote the solution of 
	\begin{equation}\label{Linfracf1f2}
		\left\{
		\begin{aligned}
			\partial_t w + (-\Delta)^\alpha w &= f_j,\quad \,\,\, (x,t)\in \mathbb{T}^2\times (0, T),\\
			w(0)&= 0,\quad \,\,\,x\in \mathbb{T}^2.\\
		\end{aligned}
		\right.
	\end{equation}  
	
	We have the following first order linearization result. As before, we use $\theta_f$ to denote the solution of (\ref{ASE}) corresponding to the source $f$.
	
	\begin{proposition}\label{p3}
		Suppose $s+\alpha> 2$. For $j= 1,2$, we have 
		$$\frac{\theta_{\epsilon f_j}}{\epsilon}\to w_j,\qquad
		\frac{\theta_{\epsilon (f_1+ f_2)}- \theta_{\epsilon f_j}}{\epsilon}\to w_{3-j}$$
		in $L^2(0, T; H^{2\alpha}(\mathbb{T}^2))\cap L^\infty(0, T; H^{\alpha}(\mathbb{T}^2))$ as $\epsilon\to 0$.
	\end{proposition}
	\begin{proof}
		Let $h= \frac{\theta_{\epsilon f_j}}{\epsilon}- w_j$, which satisfies
		\begin{equation}\label{Linfrach1}
			\left\{
			\begin{aligned}
				\partial_t h + (-\Delta)^\alpha h &= 
				-\frac{\mathcal{R}(\theta_{\epsilon f_j})\cdot \nabla \theta_{\epsilon f_j}}{\epsilon},\quad \,\,\, (x,t)\in \mathbb{T}^2\times (0, T),\\
				h(0)&= 0,\quad \,\,\,x\in \mathbb{T}^2.\\
			\end{aligned}
			\right.
		\end{equation}  
		Based on the continuous embedding 
		$$H^{s+\alpha-1}(\mathbb{T}^2)\hookrightarrow L^\infty(\mathbb{T}^2)$$
		and (\ref{Resnickeps}), we have
		\[\|\mathcal{R}(\theta_{\epsilon f_j})\cdot \nabla \theta_{\epsilon f_j}\|_{L^2(0, T; L^2(\mathbb{T}^2))}\leq C_{T, f_j}\epsilon^2.\]
		
		Then (\ref{L2est}) implies $h\to 0$ in $L^2(0, T; H^{2\alpha}(\mathbb{T}^2))\cap L^\infty(0, T; H^{\alpha}(\mathbb{T}^2))$ as $\epsilon\to 0$.
		
		Now let $$h= \frac{\theta_{\epsilon (f_1+ f_2)}- \theta_{\epsilon f_j}}{\epsilon}- w_{3-j},$$
		which satisfies
		\begin{equation}\label{Linfrach2}
			\left\{
			\begin{aligned}
				\partial_t h + (-\Delta)^\alpha h &= 
				-\frac{\mathcal{R}(\theta_{\epsilon (f_1+f_2)})\cdot \nabla \theta_{\epsilon (f_1+f_2)}-\mathcal{R}(\theta_{\epsilon f_j})\cdot \nabla \theta_{\epsilon f_j}}{\epsilon},\quad \,\,\, (x,t)\in \mathbb{T}^2\times (0, T),\\
				h(0)&= 0,\quad \,\,\,x\in \mathbb{T}^2.\\
			\end{aligned}
			\right.
		\end{equation}  
		The same argument shows that the right hand side of the equation is $O(\epsilon)$, and thus $h\to 0$ as $\epsilon\to 0$.
	\end{proof}
	
	Next, we use $v$ to denote the solution of 
	\begin{equation}\label{Linfracv}
		\left\{
		\begin{aligned}
			\partial_t v + \mathcal{R}w_1\cdot \nabla w_2+ \mathcal{R}w_2\cdot \nabla w_1+ (-\Delta)^\alpha v &=0,\quad \,\,\,(x,t)\in \mathbb{T}^2\times (0, T),\\
			v(0)&= 0,\quad \,\,\,x\in \mathbb{T}^2.\\
		\end{aligned}
		\right.
	\end{equation}  
	
	We have the following second order linearization result.
	\begin{proposition}\label{p4}
		We have 
		$$\frac{\theta_{\epsilon (f_1+ f_2)}- \theta_{\epsilon f_1}- \theta_{\epsilon f_2}}{\epsilon^2}\to v$$
		in $L^2(0, T; H^{\alpha}(\mathbb{T}^2))\cap C([0, T]; L^{2}(\mathbb{T}^2))$ as $\epsilon\to 0$.
	\end{proposition}
	\begin{proof}
		Let $$h= \frac{\theta_{\epsilon (f_1+ f_2)}- \theta_{\epsilon f_1}- \theta_{\epsilon f_2}}{\epsilon^2}- v,$$ which satisfies
		\begin{equation}\label{Linfracvw}
			\left\{
			\begin{aligned}
				\partial_t h + (-\Delta)^\alpha h &= 
				-S_\epsilon,\quad \,\,\, (x,t)\in \mathbb{T}^2\times (0, T),\\
				h(0)&= 0,\quad \,\,\,x\in \mathbb{T}^2,\\
			\end{aligned}
			\right.
		\end{equation} 
		where
		\begin{equation}
			\begin{split}
				S_\epsilon:= \frac{\mathcal{R}(\theta_{\epsilon (f_1+f_2)})\cdot \nabla \theta_{\epsilon (f_1+f_2)}-\mathcal{R}(\theta_{\epsilon f_1})\cdot \nabla \theta_{\epsilon f_1}- \mathcal{R}(\theta_{\epsilon f_2})\cdot \nabla \theta_{\epsilon f_2}}{\epsilon^2}
				- \mathcal{R}w_1\cdot \nabla w_2- \mathcal{R}w_2\cdot \nabla w_1.
			\end{split}
		\end{equation}
		
		To estimate $S_\epsilon$, we write
		$$S_\epsilon= S_{\epsilon, 1}+ S_{\epsilon, 2}+ S_{\epsilon, 3}+ S_{\epsilon, 4},$$
		where
		$$S_{\epsilon, 1}:= \frac{\mathcal{R}(\theta_{\epsilon (f_1+f_2)})- \mathcal{R}(\theta_{\epsilon f_1})}{\epsilon}\cdot \frac{\nabla (\theta_{\epsilon (f_1+f_2)}-\theta_{\epsilon f_1})}{\epsilon}- \mathcal{R}w_2\cdot \nabla w_1,$$
		$$S_{\epsilon, 2}:= \frac{\mathcal{R}(\theta_{\epsilon f_1})}{\epsilon}\cdot \frac{\nabla (\theta_{\epsilon (f_1+f_2)}-\theta_{\epsilon f_2})}{\epsilon}- \mathcal{R}w_1\cdot \nabla w_2,$$
		$$S_{\epsilon, 3}:= \frac{\mathcal{R}(\theta_{\epsilon (f_1+f_2)})- \mathcal{R}(\theta_{\epsilon f_2})}{\epsilon}\cdot \frac{\nabla \theta_{\epsilon f_2}}{\epsilon}- \mathcal{R}w_1\cdot \nabla w_2,$$
		$$S_{\epsilon, 4}:= -\frac{\mathcal{R}(\theta_{\epsilon f_1})}{\epsilon}\cdot \frac{\nabla \theta_{\epsilon f_2}}{\epsilon}+ \mathcal{R}w_1\cdot \nabla w_2.$$
		Note that we have the continuous embeddings
		$$H^{2\alpha-1}(\mathbb{T}^2)\hookrightarrow L^{\frac{1}{1-\alpha}}(\mathbb{T}^2)\hookrightarrow L^{\frac{1}{\alpha}}(\mathbb{T}^2)\hookrightarrow M(H^{\alpha}(\mathbb{T}^2)\to H^{-\alpha}(\mathbb{T}^2)),$$
		where we use $\frac{1}{2}< \alpha< 1$.
		Here $M(H^{\alpha}(\mathbb{T}^2)\to H^{-\alpha}(\mathbb{T}^2))$ denotes the space of pointwise multipliers from $H^{\alpha}(\mathbb{T}^2)$ to $H^{-\alpha}(\mathbb{T}^2))$.
		Then based on the first order linearization result, we know that $S_{\epsilon, j}$ ($j= 1, 2, 3, 4$) goes to $0$ in 
		$L^2(0, T; H^{-\alpha}(\mathbb{T}^2))$ as $\epsilon\to 0$. By (\ref{H-1est}), we conclude that $h\to 0$ in $L^2(0, T; H^{\alpha}(\mathbb{T}^2))\cap C([0, T]; L^{2}(\mathbb{T}^2))$ as $\epsilon\to 0$.
		
	\end{proof}
	
	\section{Unique continuation and Runge approximation}\label{runge}  
	We will derive a Runge approximation property based on the the unique continuation property of the fractional Laplacian in this section. These typical nonlocal features make the Calder\'on type inverse problems for fractional operators much more manageable compared with their local counterparts.
	
	The following unique continuation property of the fractional Laplacian can be viewed as a very special case of the much more general entanglement principle (see Theorem 1.8 in \cite{feizmohammadi2024calder}) on closed manifolds.
	
	\begin{proposition}\label{UCP}
		Let $r\in (0, \infty)\setminus \mathbb{N}$. Let $W\subset \mathbb{T}^2$ be nonempty and open. Suppose that $u\in H^{2r}(\mathbb{T}^2)$ satisfies
		$$(-\Delta)^r u= u= 0$$ in $W$. Then $u= 0$ in $\mathbb{T}^2$.
	\end{proposition}
	
	We remark that the unique continuation property of the fractional Laplacian in $\mathbb{R}^n$ was established earlier (see Theorem 1.2 in \cite{ghosh2020calderon}) relying on the Carleman estimates in \cite{ruland2015unique}.
	
	The Runge approximation property below will follow from the unique continuation property above. For the exterior problem for the fractional diffusion equation in $\mathbb{R}^n$,
	an analogous density result has been obtained in Section 2 in \cite{ruland2020quantitative}. Another analogous result for the fractionally damped wave equation can be found in Section 4 in \cite{li2023inverse}.
	
	\begin{proposition}\label{RAP}
		The set
		$$S:=\{u_f|_{(0, T)\times(\mathbb{T}^2\setminus W)}: f\in C^\infty_c(W\times (0, T))\}$$
		is dense in $L^2(0, T; L^2(\mathbb{T}^2\setminus W))$. Here $u_f$ is the solution of 
		\begin{equation}\label{Lin}
			\left\{
			\begin{aligned}
				\partial_t u + (-\Delta)^\alpha u &= f,\quad \,\,\, (x,t)\in \mathbb{T}^2\times (0, T),\\
				u(0)&= 0,\quad \,\,\,x\in \mathbb{T}^2.\\
			\end{aligned}
			\right.
		\end{equation}  
	\end{proposition}
	\begin{proof}
		It suffices to prove the following statement: 
		Let $g\in L^2(0, T; L^2(\mathbb{T}^2\setminus W))$. If 
		$$\int^T_0\int_{\mathbb{T}^2\setminus W}ug= 0$$ for all $u\in S$, then $g= 0$. 
		
		We consider $\tilde{g}\in L^2(0, T; L^2(\mathbb{T}^2))$ which extends $g$ by zeros in $W$, and the dual problem
		\begin{equation}\label{duallin}
			\begin{aligned}
				-\partial_t v + (-\Delta)^\alpha v&= \tilde{g},\quad \,\,\,(x,t)\in\mathbb{T}^2\times (0, T)\\			
				v(T)&= 0,\quad \,\,\,x\in \mathbb{T}^2,\\
			\end{aligned}
		\end{equation}
		which has the solution $v$ satisfying
		$$v\in L^2(0, T; H^{2\alpha}(\mathbb{T}^2))\cap L^\infty(0, T; H^{\alpha}(\mathbb{T}^2)),\qquad 
		\partial_t v\in L^2(0, T; L^2(\mathbb{T}^2)).$$
		The assumption implies 
		\begin{equation}\label{RAP1}
			0= \int^T_0\langle -\partial_t v + (-\Delta)^\alpha v, u\rangle\,\mathrm{d} t
		\end{equation}
		for $u\in S$. Based on the initial and final conditions, we integrate by parts to obtain
		
		$$0= \int^T_0\langle \partial_t u + (-\Delta)^\alpha u, v\rangle\,\mathrm{d} t=
		\int^T_0\int_{\mathbb{T}^2}fv\mathrm{d} x\mathrm{d} t= \int^T_0\int_{W}fv\mathrm{d} x\mathrm{d} t$$
		for $f\in C^\infty_c(W\times (0, T))$ since $u$ is the solution of (\ref{Lin}). Hence $v= 0$ in $W\times (0, T)$.
		Note that 
		$$-\partial_t v + (-\Delta)^\alpha v= 0$$
		in $W\times (0, T)$ since $v$ is the solution of (\ref{duallin}), so $(-\Delta)^\alpha v= 0$ in $W\times (0, T)$. By the unique continuation property, we have $v= 0$ in $\mathbb{T}^2\times (0, T)$, and thus $g= 0$.
	\end{proof}
	
	\section{Proof of the main theorem}\label{proof}
	We are ready to prove Theorem \ref{main}.
	\begin{proof}[Proof of Theorem \ref{main}]
		For $f\in C^\infty_c(W\times (0, T))$, we use $\theta^{(j)}$ ($j= 1, 2$) to denote the solution of 
		\begin{equation}\label{ASEj}
			\left\{
			\begin{aligned}
				\partial_t \theta + \mathcal{R}_j\theta\cdot \nabla \theta +(-\Delta)^\alpha \theta &= f,\quad \,\,\, (x,t)\in \mathbb{T}^2\times (0, T),\\
				\theta(0)&= 0,\quad \,\,\,x\in \mathbb{T}^2.\\
			\end{aligned}
			\right.
		\end{equation}
		The assumption $L_{\mathcal{R}_1}= L_{\mathcal{R}_2}$ implies $\theta^{(1)}-\theta^{(2)}= 0$ in $W\times (0, T)$
		and $\mathcal{R}_1\theta^{(1)}= \mathcal{R}_2\theta^{(2)}$
		in $W\times (0, T)$. Hence, based on the identity
		$$\partial_t \theta^{(1)} + \mathcal{R}_1\theta^{(1)}\cdot \nabla \theta^{(1)} +(-\Delta)^\alpha \theta^{(1)}=
		\partial_t \theta^{(2)} + \mathcal{R}_2\theta^{(2)}\cdot \nabla \theta^{(2)} +(-\Delta)^\alpha \theta^{(2)}$$
		in $W\times (0, T)$, we have
		$$(-\Delta)^\alpha (\theta^{(1)}- \theta^{(2)})= 0$$
		in $W\times (0, T)$.
		By the unique continuation property, we conclude that
		$\theta^{(1)}=\theta^{(2)}$ in $\mathbb{T}^2\times (0, T)$.
		
		Now we consider the source which has the form
		$$f= \epsilon(f_1+ f_2),\qquad f_1, f_2\in C^\infty_c(W\times (0, T)).$$
		We define $w_j$ as before to be the solution of (\ref{Linfracf1f2}). We use 
		$v^{(j)}$ ($j= 1, 2$) to denote the solution of 
		\begin{equation}\label{Linfracvj}
			\left\{
			\begin{aligned}
				\partial_t v + \mathcal{R}_jw_1\cdot \nabla w_2+ \mathcal{R}_jw_2\cdot \nabla w_1+ (-\Delta)^\alpha v &=0,\quad \,\,\,(x,t)\in \mathbb{T}^2\times (0, T),\\
				v(0)&= 0,\quad \,\,\,x\in \mathbb{T}^2.\\
			\end{aligned}
			\right.
		\end{equation}  
		The second order linearization result ensures that $\theta^{(1)}=\theta^{(2)}$ in $\mathbb{T}^2\times (0, T)$
		implies $v^{(1)}= v^{(2)}$ in $\mathbb{T}^2\times (0, T)$.
		Hence, we obtain the following identity
		$$(\mathcal{R}_1- \mathcal{R}_2)w_1\cdot \nabla w_2+ (\mathcal{R}_1- \mathcal{R}_2)w_2\cdot \nabla w_1= 0$$
		in $\mathbb{T}^2\times (0, T)$ for any $f_1, f_2\in C^\infty_c(W\times (0, T))$.
		
		Let both sides of the identity act on $\varphi\in C^\infty_c(W^e)$ and integrate with respect to the time $t$. Then we have
		$$\int^T_0\int_{W^e}[w_2(\mathcal{R}_1- \mathcal{R}_2)w_1\cdot \nabla \varphi+ w_1(\mathcal{R}_1- \mathcal{R}_2)w_2\cdot \nabla \varphi]\,\mathrm{d}x\mathrm{d}t= 0,$$
		where we use $\div(\mathcal{R}_1- \mathcal{R}_2)= 0$.
		Based on the Runge approximation property, we have 
		$$\int^T_0\int_{W^e}[\tilde{\phi}_2(\mathcal{R}_1- \mathcal{R}_2)\tilde{\phi}_1\cdot \nabla \varphi+ \tilde{\phi}_1(\mathcal{R}_1- \mathcal{R}_2)\tilde{\phi}_2\cdot \nabla \varphi]\,\mathrm{d}x\mathrm{d}t= 0$$
		for any $\tilde{\phi}_1, \tilde{\phi}_2\in C^\infty_c(W^e\times (0, T))$. In particular, we have 
		$$\int_{W^e}[\phi_2(\mathcal{R}_1- \mathcal{R}_2)\phi_1\cdot \nabla \varphi+ \phi_1(\mathcal{R}_1- \mathcal{R}_2)\phi_2\cdot \nabla \varphi]\,\mathrm{d}x= 0$$
		for any $\phi_1, \phi_2\in C^\infty_c(W^e)$. Now we write 
		$\nabla^\perp:= (-\partial_2, \partial_1)$. Then by the definition of $\mathcal{R}$ (see (\ref{u1u2})), we have
		$$\int_{W^e}[(\mathcal{K}_1- \mathcal{K}_2)\phi_1\nabla^\perp\phi_2\cdot \nabla \varphi+ (\mathcal{K}_1- \mathcal{K}_2)\phi_2\nabla^\perp\phi_1\cdot \nabla \varphi]\,\mathrm{d}x= 0,$$
		or equivalently,
		$$\int_{W^e}[(\mathcal{K}_1- \mathcal{K}_2)\phi_1(-\partial_2\phi_2\partial_1\varphi+ \partial_1\phi_2\partial_2\varphi)+ (\mathcal{K}_1- \mathcal{K}_2)\phi_2(-\partial_2\phi_1\partial_1\varphi+ \partial_1\phi_1\partial_2\varphi)]\,\mathrm{d}x= 0.$$
		
		Now we consider disjoint open subsets $W_1, W_2$ s.t. $W_1\cup W_2\subset W^e$. We consider $\phi_j\in C^\infty_c(W_j)$ ($j= 1, 2$). For chosen $\phi_1, \phi_2$, we choose $\varphi\in C^\infty_c(W_1)$ s.t. $\varphi= x_1$ or $x_2$ on $\mathrm{supp}\, \phi_1$. Then the integral identity above implies
		$$\int_{W^e}(\mathcal{K}_1- \mathcal{K}_2)\phi_2(\partial_1\phi_1)\,\mathrm{d}x= 
		\int_{W^e}(\mathcal{K}_1- \mathcal{K}_2)\phi_2(\partial_2\phi_1)\,\mathrm{d}x
		=0,$$
		or equivalently,
		$$\int_{W_1}\int_{W_2}\partial_1(K_1- K_2)(x-y)\phi_2(y)\phi_1(x)\,\mathrm{d}y\mathrm{d}x$$
		$$=\int_{W_1}\int_{W_2}\partial_2(K_1- K_2)(x-y)\phi_2(y)\phi_1(x)\,\mathrm{d}y\mathrm{d}x
		=0.$$
		Since our choices of $\phi_1, \phi_2$ are arbitrary, we conclude that 
		$$\partial_1(K_1- K_2)=
		\partial_2(K_1- K_2)= 0$$
		in the set 
		$$\{x-y: x\in W_1, y\in W_2\}.$$
		Since our choices of $W_1, W_2$ are arbitrary, we conclude that 
		$$\partial_1(K_1- K_2)=
		\partial_2(K_1- K_2)= 0$$
		in the set 
		$$\{x-y: x, y\in W^e, x\neq y\},$$
		which implies $$\mathcal{R}_1g|_{W^e}= \mathcal{R}_2g|_{W^e},\qquad g\in 
		C^\infty_c(W^e).$$
	\end{proof}
	
	\section{Acknowledgments}
	L.L. would like to thank Professor Gunther Uhlmann for helpful discussions. W.W. was partially supported by the Simons Foundation TSM grant. The authors would like to thank the organizers of the CBMS Conference: Inverse Problems and Nonlinearity at Clemson University, where part of this work was completed.
	
	\section{Appendix}\label{appen}
	We will sketch the proof of Proposition \ref{p1} here.
	We refer readers to Chapter 3 in \cite{Resnick1996} for more detailed arguments in the setting of the SQG equation.
	For convenience, we will use the following notations
	$$\|\cdot\|_p:= \|\cdot\|_{L^p(\mathbb{T}^2)},\qquad
	|\cdot|_s:= \|\cdot\|_{H^s(\mathbb{T}^2)},\qquad 
	\Lambda:= (-\Delta)^\frac{1}{2}.$$
	
	\begin{proof}[Proof of Proposition \ref{p1}]
		(i)\, \textbf{Existence and regularity}: We need to consider a sequence of retarded mollifications of the nonlinear active scalar equation. 
		
		We pick a $\phi \in C^\infty_c(0, \infty), \phi \geq 0$ satisfying supp $\phi \subset[1,2]$ and $\int_{0}^{\infty} \phi(\tau) \,\mathrm{d} \tau=1$, and let
		$$
		U_{\delta}[\theta](t):= \int_{0}^{\infty} \phi(\tau) \mathcal{R} \theta(t-\delta \tau) \,\mathrm{d}\tau,
		$$
		where we set $\theta(t)= 0$ for $t<0$. Thus, $U_{\delta}[\theta](t)$ depends on the values of $\theta(t^{\prime})$ only for $t^{\prime} \in[t-2 \delta, t-\delta]$.
		Since $\mathcal{R}$ is bounded on $L^{q}(\mathbb{T}^2)$ and $H^{s}(\mathbb{T}^2)$, for all $t>0$ we have
		
		\begin{equation*}
			\sup _{[0, t]}\left\|U_{\delta}[\theta]\right\|_{q} \leqslant C \sup _{[0, t]}\|\theta\|_{q}~\text{and}~
			\int_{0}^{t}\left|U_{\delta}[\theta]\right|_{s}^{2}\,\mathrm{d}\tau \leqslant \int_{0}^{t}|\theta|_{s}^{2}\,\mathrm{d}\tau.
		\end{equation*}
		
		Now we take a sequence $\delta_{n} \rightarrow 0^+$ and consider the equation
		$$
		\partial_t \theta_{n}+ u_{n}\cdot \nabla \theta_{n}+\Lambda^{2 \alpha} \theta_{n}=f
		$$
		with initial data $\theta_{n}(0)= 0$ and $u_{n} :=U_{\delta_{n}}\left[\theta_{n}\right]$. The advantage of this mollification is that the equation is actually a linear passive scalar equation on each subinterval $\left[t_{k}, t_{k+1}\right]$ with $t_{k}=k \delta_{n}$, since $u_{n}$ is determined by the values of $\theta_{n}$ on the two previous subintervals. 
		This enables us to inductively apply the results for the linear equation to obtain the existence of $\theta_{n}$ and its $L^{q}$ bounds.
		
		To obtain the $H^{s}$-estimate, we let both sides of the equation act on $\Lambda^{2s}\theta_{n}$. We have
		$$
		\frac{1}{2} \frac{d}{d t}\left|\theta_{n}\right|_{s}^{2}+\left|\theta_{n}\right|_{s+\alpha}^{2} \leqslant\left|\int_{\mathbb{T}^{2}} \Lambda^{s} \theta_{n} \Lambda^{s}\left(u_{n}\cdot \nabla \theta_{n}\right) d x\right|+\left|\int_{\mathbb{T}^{2}} \Lambda^{s} f \Lambda^{s} \theta_{n} d x\right| .
		$$
		Based on the Sobolev embedding, the H{\"o}lder inequality and the product lemma 
		$$\|\Lambda^{s+1-\alpha}(u\theta)\|_2
		\leq C(\|u\|_q\|\Lambda^{s+1-\alpha}\theta\|_2+ \|\theta\|_q\|\Lambda^{s+1-\alpha}u\|_2)$$
		(see Lemma A.4 in \cite{Resnick1996}), we can further estimate that
		$$
		\left|\int_{\mathbb{T}^{2}} \Lambda^{s} f \Lambda^{s} \theta_{n} d x\right|
		\leqslant
		\frac{1}{8}\left|\theta_{n}\right|_{s+\alpha}^{2}+C|f|_{s-\alpha},
		$$
		$$\left|\int_{\mathbb{T}^{2}} \Lambda^{s} \theta_{n} \Lambda^{s}\left(u_{n}\cdot \nabla \theta_{n}\right) d x\right|
		\leqslant \frac{1}{4}\left|\theta_{n}\right|_{s+\alpha}^{2}+\frac{1}{8}\left|u_{n}\right|_{s+\alpha}^{2}
		+C[\left\|u_{n}\right\|_{q}^N\left|\theta_{n}\right|_{s}^{2}+\left\|\theta_{n}\right\|_{q}^N\left|u_{n}\right|_{s}^{2}],
		$$
		where $N:= \alpha/(\alpha-1/2-1/q)$. 
		Then we integrate over $[0, t]$ to obtain the
		Gr\"onwall type integral inequality
		$$|\theta_n(t)|^2_{s}+ \int^t_0 |\theta_n(\tau)|^2_{s+\alpha}\,\mathrm{d}\tau$$
		$$\leq C(\sup_{\tau\in[0, t]}\|\theta_n(\tau)\|_q)^N\int^t_0 |\theta_n(\tau)|^2_{s}\,\mathrm{d}\tau
		+ C\int^t_0 |f(\tau)|^2_{s-\alpha}\,\mathrm{d}\tau,$$
		which enables us to obtain the desired estimate for $\theta_n$.
		
		Finally, we can use a compactness argument to obtain $\theta$ and the desired estimate.
		
		(ii)\, \textbf{Uniqueness}: Suppose $\theta_{j}$ ($j=1,2$) are two solutions. Then $\tilde{\theta}:= \theta_{1}-\theta_{2}$ satisfies
		$$
		\partial_t \tilde{\theta}+\tilde{u}\cdot \nabla \theta_{1}+ u_{2} \cdot\nabla \tilde{\theta}+\Lambda^{2 \alpha} \tilde{\theta}=0,
		$$
		where $\tilde{u}:=\mathcal{R} \tilde{\theta}$. 
		Let both sides of the equation act on 
		$\tilde{\psi}:= -\mathcal{K} \tilde{\theta}$.  Note that
		$$
		\int_{\mathbb{T}^{2}} \tilde{\psi}(\tilde{u}\cdot \nabla \theta_{1})\,\mathrm{d}x=0 
		$$
		since $\tilde{\psi}\tilde{u}$ can be written as a divergence-free term.
		Then by the Sobolev embedding and the Hölder inequality, we have
		\begin{equation*}
			\begin{split}
				\frac{1}{2} \frac{d}{d t}|\widetilde{\psi}|_{\frac{1}{2}}^{2}+|\tilde{\psi}|_{\frac{1}{2}+\alpha}^{2} \leqslant\left|\int_{\mathbb{T}^{2}} \tilde{\theta} (u_{2} \cdot\nabla \tilde{\psi}) \,\mathrm{d}x\right|
				&\leqslant
				\left\|u_{2}\right\|_{q}\|\tilde{\theta}\|_{p}\|\nabla \tilde{\psi}\|_{p} 
				\\&\leqslant C\left\|u_{2}\right\|_{q}|\widetilde{\theta}|_{\frac{1}{q}}|\nabla \tilde{\psi}|_{\frac{1}{q}} \leqslant C\left\|\theta_{2}\right\|_{q}|\tilde{\psi}|_{1+\frac{1}{q}}^{2}
				\leqslant \frac{1}{2}|\widetilde{\psi}|_{\frac{1}{2}+\alpha}^{2}+C\left\|\theta_{2}\right\|_{q}^{N}|\tilde{\psi}|_{\frac{1}{2}}^{2},
			\end{split}
		\end{equation*}
		where $p$ is given by $1/q+ 2/p= 1$.
		
		Hence, we obtain the Gr\"onwall type inequality
		$$
		\frac{d}{d t}|\tilde{\psi}|_{\frac{1}{2}}^{2}+|\tilde{\psi}|_{\frac{1}{2}+\alpha}^{2} \leqslant C\left\|\theta_{2}\right\|_{q}^{N}|\tilde{\psi}|_{\frac{1}{2}}^{2},
		$$
		which implies $\tilde{\psi}= 0$, and thus 
		$\tilde{\theta}= 0$.
	\end{proof}

	\bigskip\noindent 
	Li Li, Yau Mathematical Sciences Center, Tsinghua University, Beijing, China;\\
	e-mail: \url{lili19940301@mail.tsinghua.edu.cn}
	
	\bigskip\noindent 
	Weinan Wang, Department of Mathematics, University of Oklahoma, Norman, OK, USA;\\
	e-mail: \url{ww@ou.edu}
\end{document}